\documentclass[a4paper]{amsart}

\pdfoutput=1

\usepackage{lmodern}
\usepackage{bbm}
\usepackage{amsmath,amssymb,amsthm,url,scalerel}
\usepackage[utf8]{inputenc}
\usepackage[T1]{fontenc}
\usepackage{libertine}
\usepackage[libertine,cmintegrals,cmbraces]{newtxmath}
\usepackage{verbatim}
\usepackage[shortlabels]{enumitem}
\usepackage{stmaryrd}
\usepackage{mathtools}
\usepackage{microtype}
\usepackage{multicol}
\usepackage{xcolor}
\definecolor{darkred}{RGB}{139,0,0}
\definecolor{darkblue}{RGB}{0,0,139}
\definecolor{darkgreen}{RGB}{0,100,0}
\usepackage{hyperref}
\hypersetup{
  colorlinks   = true, 
  urlcolor     = darkred, 
  linkcolor    = darkred, 
  citecolor   = darkgreen}
\usepackage{tikz}
\usepackage{tikz-cd}
\usepackage{sseq}
\usepackage[enableskew]{youngtab}
\usepackage{ytableau}
\usepackage[capitalise]{cleveref}
\usepackage{nicefrac}
\usepackage[margin=1.3in]{geometry}

\AtBeginDocument{%
   \def\MR#1{}
}
\setenumerate[1]{label=(\roman*),} 
\setitemize[1]{label=\raisebox{0.25ex}{\tiny$\bullet$}}

\newcommand{\Diff}{\ensuremath{\mathrm{Diff}}}

\newcommand{\BDiff}{\ensuremath{\mathrm{BDiff}}}

\newcommand{\Emb}{\ensuremath{\mathrm{Emb}}}

\newcommand{\id}{\mathrm{id}}

\newcommand{\im}{\mathrm{im}}

\newcommand{\inc}{\mathrm{inc}}

\newcommand{\colim}{\mathrm{colim}}
\newcommand{\hocolim}{\mathrm{hocolim}}
\newcommand{\Aut}{\mathrm{Aut}}

\newcommand{\Fr}{\mathrm{Fr}}

\newcommand{\interior}{\mathrm{int}}

\newcommand{\SO}{\mathrm{SO}}

\DeclareMathAlphabet{\mathpzc}{OT1}{pzc}{m}{it}

\newcommand{\oH}{\ensuremath{\mathrm{H}}}

\newcommand{\bfC}{\ensuremath{\mathbf{C}}}

\newcommand{\bfZ}{\ensuremath{\mathbf{Z}}}

\newcommand{\Tor}{{\mathrm{Tor}}}

\newcommand{\ra}{\rightarrow}

\newcommand{\lra}{\longrightarrow}

\newcommand{\xra}[1]{\xrightarrow{#1}}

\newcommand{\longtwoheadrightarrow}{\relbar\joinrel\twoheadrightarrow}

\newcommand{\mr}[1]{{\mathrm{#1}}}

\newcommand{\circled}[1]{\raisebox{.5pt}{\textcircled{\raisebox{-.9pt} {#1}}}}
\makeatletter
\renewcommand{\boxed}[1]{\text{\fboxsep=.2em\fbox{\m@th$\displaystyle#1$}}}
\makeatother

\newtheorem{bigthm}{Theorem}

\newtheorem{thm}{Theorem}

\newtheorem{lem}[thm]{Lemma}

\theoremstyle{definition}

\theoremstyle{remark}

\newtheorem{rem}[thm]{Remark}
\newtheorem*{nrem}{Remark}

\calclayout

\begin{document}

\title{On Torelli groups and Dehn twists of smooth 4-manifolds}

\author{Manuel Krannich}
\address{Department of Mathematics, Karlsruhe Institute of Technology, 76131 Karlsruhe, Germany}
\email{krannich@kit.edu}

\author{Alexander Kupers}
\address{Department of Computer and Mathematical Sciences, University of Toronto Scarborough, 1265 Military Trail, Toronto, ON M1C 1A4, Canada}
\email{a.kupers@utoronto.ca}

%\subjclass[2010]{57R50, 11F06, 20E26}
\begin{abstract}This note has two related but independent parts. Firstly, we prove a generalisation of a recent result of Gay on the smooth mapping class group of $S^4$. Secondly, we give an alternative proof of a consequence of work of Saeki, namely that the Dehn twist along the boundary sphere of a simply-connected closed smooth $4$-manifold $X$ with $\partial X\cong S^3$ is trivial after taking connected sums with enough copies of $S^2\times S^2$.
\end{abstract}


\maketitle

This note serves to record two results on the mapping class group of compact simply connected smooth $4$-manifolds, the first on their so-called \emph{Torelli subgroup} and the second on stable \emph{Dehn twists}.

\subsection*{Generating Torelli groups}Recently, Gay \cite[Theorem 1]{Gay} constructed a surjection
\begin{equation}\label{equ:gays-morphism}\smash{
\underset{g \to \infty}{\colim}\, \pi_1(\Emb(\sqcup^g S^2,W_{g,1}),\inc) \longtwoheadrightarrow \pi_0(\Diff^+(S^4))}\end{equation}
onto the smooth oriented mapping class group of $S^4$, whose domain is a colimit of the fundamental groups of the space of smooth embeddings of $\sqcup^g S^2$ into $W_{g,1}\coloneq (S^2 \times S^2)^{\sharp g} \setminus \mr{int}(D^4)$, based at the embedding $\inc \colon \sqcup^g S^2 \hookrightarrow W_{g,1}$ induced by $g$ copies of the inclusion $S^2 \times \{\ast\} \subset S^2 \times S^2$. In the first part of this note, we use work of Kreck \cite{KreckIsotopy} and Quinn \cite{Quinn} to prove a generalisation of Gay's result which applies to all simply connected closed oriented smooth 4-manifolds $X$. For such a general $X$, the target of \eqref{equ:gays-morphism} needs to be replaced by the \emph{Torelli subgroup} $\Tor(X) \subset \Diff^+(X)$---the subgroup of those diffeomorphisms that act as the identity on the integral homology of $X$. 

\begin{bigthm}\label{thm:gay}For a closed simply-connected smooth 4-manifold $X$, there is a surjective homomorphism
	\[\smash{\underset{g \to \infty}{\colim}\, \pi_1(\Emb(\sqcup^gD^2\times S^2,X \sharp W_{g,1}),\inc) \longtwoheadrightarrow \pi_0(\Tor(X)).}\]
\end{bigthm}

\begin{nrem} For $X=S^4$, we have $\Tor(S^4)=\Diff^+(S^4)$ and our result recovers Gay's. Indeed, comparing the construction we give in the proof of \cref{thm:gay} with Gay's description of \eqref{equ:gays-morphism} in \cite[p.\,8--9]{Gay}, one sees the homomorphism we construct agrees with Gay's, up to precomposition with the surjective homomorphism on fundamental groups induced by the map $\Emb(\sqcup^gD^2\times S^2, W_{g,1})\ra \Emb(\sqcup^gS^2,W_{g,1})$ obtained by restriction along $\sqcup^g*\times S^2\subset \sqcup^gD^2\times S^2$. The fact that this homomorphism is surjective follows from the long exact sequence in homotopy groups induced by the restriction map, since the homotopy fibre over the inclusion is homotopy equivalent to the path-connected space $\mr{Map}(S^2,\mr{SO}(2))^g$.\end{nrem}

\subsection*{Dehn twists are stably trivial}Related to, but independent of, the proof of \cref{thm:gay} (see \cref{rem:related}), we show in the second part of this note that the Dehn twist $t_X\in\pi_0(\Diff_\partial(X\sharp D^4))$ along the boundary sphere of $X\sharp D^4$ for a closed smooth simply-connected $4$-manifold is stably isotopic to the identity. In addition to Kreck's and Quinn's work mentioned above, this relies on Wall's stable classification in \cite{Wall4mfds} and Galatius--Randal-Williams' work on stable moduli spaces of manifolds \cite{GRWII}. 

\begin{bigthm}[Saeki] \label{thm:dehn}For a closed simply-connected smooth 4-manifold $X$, the Dehn twist $t_X \in \pi_0(\Diff_\partial(X\sharp D^4))$ lies in the kernel of the stabilisation map $\smash{\pi_0(\Diff_\partial(X\sharp D^4))\ra \underset{g \to \infty}{\colim}\, \pi_0(\Diff_\partial(X\sharp W_{g,1}))}$.
\end{bigthm}

\begin{nrem}\,
\begin{enumerate}[leftmargin=*]
\item As indicated by the attribution, \cref{thm:dehn} can also be deduced (by very different means) from Saeki's computation of the stable mapping class groups of compact simply-connected smooth $4$-manifolds with nonempty connected boundary \cite[Thm 3.7, Prop.\,4.2]{Saeki}.
\item In the case where $X$ does not not admit a spin structure, there is direct argument that the Dehn twist $t_X$ is isotopic to the identity, even before stabilisation by $S^2\times S^2$; see \cite[Corollary A.5]{OrsonPowell}. In the special case where $X=Y\sharp \bfC P^2$ or $\smash{X=Y\sharp \overline{\bfC P}^2}$ for some $Y$, this was shown in \cite[Theorem 2.4]{Giansiracusa}.
\item If $X$ is spin, then the Dehn twist $t_X$ is \emph{not} isotopic to the identity \emph{without} stabilisation in general, for example in the case $X=K3$. This was independently shown by Baraglia--Konno \cite{BaragliaKonno2} and Kronheimer--Mrowka \cite[Proposition 1.2]{KronheimerMrowka} (see also the discussion in \cite[Section 1.1]{OrsonPowell}).
\end{enumerate}
\end{nrem}

We now turn to the proofs of Theorem~\ref{thm:gay} and \ref{thm:dehn}, in this order.

\subsection*{Proof of \cref{thm:gay}}\label{sec:proof}
Consider the embedding $\smash{\sqcup^g D^2\times S^2\subset W_{g,1}=(S^2 \times S^2)^{\sharp g}\backslash\interior(D^4)}$ obtained by viewing $D^2\subset S^2$ as the upper hemisphere, performing the connected sums in the complement of $\sqcup^g D^2\times S^2\subset \sqcup^g S^2\times S^2$, and choosing the embedded $4$-disc also away from this complement. Note that the complement $Y_{g,1}\coloneq W_{g,1}\backslash \sqcup^g \mr{int}(D^2\times S^2)$ has the diffeomorphism type of the connected sum $D^4\sharp\smash{(D^2\times S^2)^{\sharp g}}$. 

\medskip

\noindent The homomorphism in \cref{thm:gay} will arise as the middle row of a commutative diagram of groups
\[\hspace{-0.2cm}\begin{tikzcd}[column sep=0.6cm]
&\pi_0(\Tor_\partial(X\sharp D^4))\arrow[dr,equal, bend left=10]\arrow[d,"\circled{4}",hook]&&\\
\underset{g \to \infty}{\colim}\, \pi_1(\Emb(\sqcup^g D^2\times S^2,X \sharp W_{g,1})) \rar{\circled{1}} &\underset{g \to \infty}{\colim}\, \pi_0(\Tor_\partial(X \sharp Y_{g,1}))\arrow[r,"\circled{2}",two heads]\dar{\circled{5}}&\pi_0(\Tor_\partial(X \sharp D^4))\arrow[r,two heads,"\circled{3}"]& \pi_0(\Tor(X))\\
&\underset{g \to \infty}{\colim}\, \pi_0(\Diff_\partial(X\sharp W_{g,1}))&&
\end{tikzcd}\]
which we explain in the following.
\subsubsection*{\circled{1} \& \circled{5}}By the parametrised isotopy extension theorem (see e.g.~\cite[Theorem 6.1.1]{WallDiffTop}) restricting diffeomorphisms of $X\sharp W_{g,1}$ along the embedding $\sqcup^g D^2\times S^2\subset X\sharp W_{g,1}$ induces a fibre sequence 
\begin{equation}\label{equ:fibration}
	\Diff_\partial(X \sharp Y_{g,1})\lra \Diff_\partial(X\sharp W_{g,1})\lra \Emb(\sqcup^g D^2\times S^2,X\sharp W_{g,1})
\end{equation}
with fibre taken over the inclusion. Extending diffeomorphisms and embeddings by the identity along the inclusion of pairs $(W_{g-1,1},\sqcup^{g-1} D^2\times S^2)\subset (W_{g,1},\sqcup^{g} D^2\times S^2)$ induces a map of fibre sequences between \eqref{equ:fibration} for $g-1$ and $g$. The long exact sequences of homotopy groups then induce an exact sequence \begin{equation}\label{equ:exact-colim-sequence}
	\smash{\underset{g \to \infty}{\colim}\,\pi_1(\Emb(\sqcup^g D^2\times S^2,X \sharp W_{g,1}),\inc) \lra \underset{g \to \infty}{\colim}\,\pi_0(\Diff_\partial(X \sharp Y_{g,1}))\lra \underset{g \to \infty}{\colim}\,\pi_0(\Diff_\partial(X \sharp W_{g,1})).}
\end{equation}
The right morphism in this sequence gives $\circled{5}$ by restriction to the Torelli subgroups in the source and the left morphism yields a homomorphism that qualifies as \circled{$1$} by the following lemma.
\begin{lem}For a closed simply-connected smooth $4$-manifold $X$, the image of the connecting map 
	\[\pi_1(\Emb_{}(\sqcup^g D^2\times S^2,X\sharp W_{g,1}),\mr{inc})\lra \pi_0(\Diff_\partial(X\sharp Y_{g,1}))\] 
induced by \eqref{equ:fibration} is contained in the subgroup $\pi_0(\Tor_\partial(X\sharp Y_{g,1}))\subset \pi_0(\Diff_\partial(X\sharp Y_{g,1}))$.
\end{lem} 

\begin{proof}The inclusion $Y_{g,1}\subset W_{g,1}$ is injective on homology, so the same holds for the inclusion $X\sharp Y_{g,1}\subset X\sharp W_{g,1}$. Combining this with the observation that the composition $\pi_1(\Emb_{}(\sqcup^g D^2\times S^2,X\sharp W_{g,1}),\mr{inc})\ra \pi_0(\Diff_\partial(X\sharp Y_{g,1}))\ra \pi_0(\Diff_\partial(X\sharp W_{g,1}))$
is trivial by exactness, the claim follows.
\end{proof}

\subsubsection*{\circled{2} \& \circled{4}}The effect of gluing $\sqcup^g S^1\times D^3$ to $Y_{g,1}$ using the canonical identification $\partial(\sqcup^g S^1\times D^3)=\partial(\sqcup^g D^2\times S^2)=\partial Y_{g,1}$ agrees with the result of doing surgery along the embeddings constituting $\sqcup^gD^2\times S^2\subset W_{g,1}$, so this process yields a manifold diffeomorphic to $D^4$. Hence there are embeddings $Y_{g,1}= W_{g,1}\backslash \sqcup^g \mr{int}(D^2\times S^2)\hookrightarrow D^4$, one for each $g$, which we may choose so that the extension map $ \pi_0(\Diff_\partial(X \sharp Y_{g,1}))\ra \pi_0(\Diff_\partial(X \sharp D^4))$ is compatible with the extension maps $\pi_0(\Diff_\partial(X \sharp Y_{g-1,1}))\ra \pi_0(\Diff_\partial(X \sharp Y_{g,1}))$. Taking colimits and restricting to Torelli subgroups defines $\circled{2}$. The morphism \circled{4} is the canonical map into the colimit (note that $D^4=Y_{0,1}$). Going through the definition, one sees that the composition $\circled{2}\circ \circled{4}$ agrees with the identity, so $\circled{2}$ is split surjective and $\circled{4}$ is split injective.

\subsubsection*{\circled{3}}This morphism is induced by extension with the identity along $X\sharp D^4\subset X\sharp S^4\cong X$. We claim that it is surjective and its kernel is generated by the Dehn twist $t_X$. To see this, we use that the restriction map $\Tor(X) \ra \Emb^+(D^4,X)$ along an embedded disc $D^4\subset X$ is a fibration by parametrised isotopy extension, whose fibre over the inclusion is $\Tor_\partial(X\backslash\interior(D^4))\cong \Tor_\partial(X\sharp D^4)$. Taking derivatives induces a homotopy equivalence $\Emb^+(D^4,X)\simeq \Fr^+(X)$ to the oriented frame bundle which in turns fits into a fibre sequence $\SO(4)\ra \Fr^+(X)\ra X$. As $X$ is simply-connected, the long exact sequence of homotopy groups yields $\pi_0(\Fr^+(X))=*$ and $\pi_1(\Fr^+(X))\cong\bfZ/2$, so the long exact sequence of homotopy groups for the fibre sequence $\Tor_\partial(X\sharp D^4)\ra\Tor(X) \ra \Emb^+(D^4,X)$ induces a short exact sequence of the form
\begin{equation}\label{equ:sequence-Dehn-twist}
	\smash{\bfZ/2 \xra{t_X} \pi_0(\Diff_\partial(X\sharp D^4)) \lra \pi_0(\Diff^+(X)) \lra 0,}
\end{equation}
whose first map maps the generator to the Dehn twist $t_X$. Passing to Torelli subgroups, the claim follows.

\medskip

In addition to the above diagram, the main ingredient in the proof of \cref{thm:gay} is the following.

\begin{lem}\label{lem:generating-tor}For a closed simply-connected smooth $4$-manifold $X$, the group $\pi_0(\Tor_\partial(X\sharp D^4))$ is generated by the Dehn twist $t_X$ and the kernel of the stabilisation $(\circled{5}\circ\circled{4})\colon \pi_0(\Tor_\partial(X\sharp D^4))\to \colim_{g \to \infty}\,\pi_0(\Tor_\partial(X\sharp W_{g,1}))$.
\end{lem}

\begin{proof}
By a result of Kreck \cite[Theorem 1]{KreckIsotopy}, the subgroup $\pi_0(\Tor(X))\subset \pi_0(\Diff(X))$ agrees with the subgroup of diffeomorphisms that are pseudo-isotopic to the identity, so it follows that the group $\pi_0(\Tor_\partial(X\sharp D^4))$ is generated by the Dehn twist $t_X\in \pi_{0}(\Tor_\partial(X\sharp D^4))$ and the subgroup of $\pi_{0}(\Tor_\partial(X\sharp D^4))$ consisting of diffeomorphisms pseudo-isotopic to the identity. But by a result of Quinn \cite[Theorem 1.4]{Quinn} (whose proof was corrected in \cite[Theorem 1.1]{GGHKS}), any diffeomorphism of $X\sharp D^4$ that is pseudo-isotopic to the identity lies in the kernel of the map $\pi_0(\Diff_\partial(X\sharp D^4))\to \colim_{g \to \infty}
\,\pi_0(\Diff_\partial(X\sharp W_{g,1}))$.
\end{proof}

We now conclude the proof of \cref{thm:gay} by showing that the composition $\circled{3}\circ\circled{2}\circ\circled{1}$ is surjective. Since $\circled{3}$ is surjective and its kernel is generated by $t_X$, it suffices to show that $\pi_0(\Tor_\partial(X\sharp D^4))$ is generated by the Dehn twist and the image of $\circled{2}\circ \circled{1}$. By \cref{lem:generating-tor} it thus suffices to show that the kernel of $\circled{5}\circ\circled{4}$ is contained in the image of $\circled{2}\circ \circled{1}$. But for $x\in \ker(\circled{5}\circ\circled{4})$, we have $\circled{4}(x)\in\im(\circled{1})$ by exactness of \eqref{equ:exact-colim-sequence}, so by applying $\circled{2}$ and using $\circled{2}\circ\circled{4}=\id$, we get $x\in\im(\circled{2}\circ\circled{1})$ as claimed.

\begin{rem}\label{rem:related}
\cref{thm:dehn} shows that the Dehn twist generator $t_X\in \pi_0(\Tor_\partial(X\sharp D^4))$ in \cref{lem:generating-tor} is not actually necessary. This also shows that the composition of $\circled{$1$}$ and $\circled{$2$}$ is already surjective, so one can replace $\pi_0(\Tor(X))$ by $\pi_0(\Tor_\partial(X\sharp D^4))$ in the statement of \cref{thm:gay}.
\end{rem}

\subsection*{Proof of \cref{thm:dehn}} Recall that the Dehn twist $t_X$ is the image of the non-trivial element under the leftmost map in \eqref{equ:sequence-Dehn-twist}. To study its behaviour under stabilisation, we use that the Dehn twists in $\pi_0(\Diff_\partial(W_{1,1}\sharp D^4))$ around the two boundary components are isotopic (see \cite[Proposition 3.7]{Giansiracusa}), so the stabilisation map $\smash{\pi_0(\Diff_\partial(X \sharp W_{g,1})) \to \pi_0(\Diff_\partial(X \sharp W_{g+1,1}))}$ sends $t_{X\sharp W_g}$ to $t_{X\sharp W_{g+1}}$. Combined with \eqref{equ:sequence-Dehn-twist} and writing $W_g\coloneq\sharp^gS^2\times S^2$ this shows that there is a stabilisation map $\pi_0(\Diff^+(X \sharp W_{g})) \to \pi_0(\Diff^+(X \sharp W_{g+1}))$, a well-defined element $t_{X\sharp W_\infty} \in \colim_{g \to \infty}\,\pi_0(\Diff_\partial(X \sharp W_{g,1}))$, an exact sequence
	\begin{equation}\label{stable-Dehn}\smash{\bfZ/2 \xra{t_{X\sharp W_\infty}} \underset{g \to \infty}{\colim}\,\pi_0(\Diff_\partial(X\sharp W_{g,1})) \lra \underset{g \to \infty}{\colim}\,\pi_0(\Diff^+(X\sharp W_g)) \lra 0,}\end{equation}
and that \cref{thm:dehn} is equivalent to showing that $t_{X\sharp W_\infty}$ is trivial. Since every orientation-preserving diffeomorphism between oriented manifolds can be assumed to fix a codimension $0$ disc, this question depends only on the oriented diffeomorphism type of $X$ up to taking connected sums with $S^2 \times S^2$, the \emph{stable diffeomorphism type}. The latter is by \cite[Theorem 2 and 3]{Wall4mfds} determined by the \emph{stable isomorphism type} of the intersection form, i.e.\,the isomorphism type of the intersection form up to adding copies of the hyperbolic form $H \coloneq (\bfZ^2,\left[\begin{smallmatrix} 0 & 1\\ 1 & 0 \end{smallmatrix}\right])$, the intersection form of $S^2\times S^2$.

If $X$ is not spin, then the Dehn twist $t_X$ is trivial even before stabilisation (see the remark in the introduction). If $X$ is spin, then its intersection form is even and its signature is divisible by $16$ as a consequence of Rokhlin's theorem. From the classification of indefinite unimodular symmetric bilinear forms over $\bfZ$ \cite[Theorem II.5.3]{MilnorHusemoller}, it follows that the stable isomorphism type (i.e.~isomorphism up to addition of $H$) of the intersection form of $X$ agrees with that of $\smash{(E_8^{\oplus -2} \oplus H^{\oplus 3})^{\oplus k}}$ for some $k\in\bfZ$, which is the intersection form of $\smash{\sharp^kK_3}$ (the convention is that $\smash{\sharp^kK_3}$ is for negative $k$ the $|k|$-fold connected sum of $K_3$ with the opposite orientation). We may thus assume $\smash{X=\sharp^kK_3}$ for some $k\in\bfZ$.
To settle this case, we first prove the following result on stable abelianisations.

\begin{lem}\label{lem:abelianisation-lemma}Let $X=\sharp^k K_3$ for $k\in \bfZ$, then
\[\underset{g \to \infty}{\colim}\, \oH_1(\pi_0(\Diff_\partial(X \sharp W_{g,1})))\cong\begin{cases}
\underset{g \to \infty}{\colim}\, \oH_1(\pi_0(\Diff_\partial(W_{g,1})))&\text{if }t_{X\sharp W_\infty}\text{ is trivial,}\\
\underset{g \to \infty}{\colim}\, \oH_1(\pi_0(\Diff_\partial(W_{g,1})))\oplus \bfZ/2&\text{otherwise}.
\end{cases}
\]
Moreover, we have $\underset{g \to \infty}{\colim}\, \oH_1(\pi_0(\Diff_\partial(W_{g,1})))\cong\bfZ/2\oplus \bfZ/2$.
\end{lem}
\begin{proof}To ease the notation, given a sequence of groups $ G_1\ra G_{2}\ra \cdots$ we write $G_\infty$ for its colimit. We first show that the following composition induces a homology isomorphism in all degrees
\begin{equation}\label{equ:comp}
	\pi_0(\Diff_\partial(W_{\infty,1}))\lra\pi_0(\Diff_\partial(X\sharp W_{\infty,1}))\lra \pi_0(\Diff^+(X\sharp W_\infty)),
\end{equation}
so in particular the second map induces a split surjection on all homology groups. The second map is as in \eqref{stable-Dehn} and the first map is constructed as follows: in terms of the identification of $X\sharp W_{g,1}$ with  $(X \sharp D^4) \natural W_{g,1}$, the stabilisation map is given by $(-)\natural \id_{W_{1,1}}$, so the map $\id_{X \sharp D^4}\natural(-)$ induces a map on colimits. 

Now by \cite[Proposition 3.7]{Giansiracusa}, the Dehn twist $t_{W_g}\in \pi_0(\Diff_\partial(W_{g,1}))$ is trivial for all $g\ge0$, so $\pi_0(\Diff_\partial(W_{g,1})) \to \pi_0(\Diff^+(W_{g}))$ is an isomorphism. Similar to the argument in \cref{lem:generating-tor}, a combination of \cite[Theorem 1, Remark 2)]{KreckIsotopy} and Quinn \cite[Theorem 1.4]{Quinn} thus identifies the composition \eqref{equ:comp} with the induced map on automorphism groups of intersection forms,
\begin{equation}\label{equ:comp-aut} \smash{\id_{E_8^{\oplus k}} \oplus (-) \colon \Aut(H^{\oplus \infty}) \lra \Aut(E_8^{\oplus k} \oplus H^{\oplus \infty})}.\end{equation}
 Fixing an isomorphism $E_8 \oplus -E_8 \cong H^{\oplus 8}$, the composition of $E_8^{\oplus k}\oplus(-)$ with $-E_8^{\oplus k}\oplus (-)$ agrees on each finite stage of the colimit with $(-)\oplus H^{\oplus 8k}$ up to an inner automorphism. As inner automorphisms induce the identity on group homology \cite[Proposition II.6.2]{Brown} and taking homology of a group preserves sequential colimits by p.\,121 loc.cit., it follows that the map $\smash{\id_{-E_8^{\oplus k}} \oplus (-) \colon\Aut(E_8^{\oplus k} \oplus H^{\oplus \infty}) \to \Aut(H^{\oplus 8+\infty}) =\Aut(H^{\oplus \infty})}$ is inverse to \eqref{equ:comp-aut} on the level of homology groups, so the composition \eqref{equ:comp} is indeed a homology isomorphism. 

If $t_{X\sharp W_\infty}$ is trivial, then the second map in \eqref{equ:comp} is an isomorphism by \eqref{stable-Dehn}, so the first map induces an isomorphism on all homology groups, in particular on $\oH_1(-)$ as claimed. If $t_{X\sharp W_\infty}$ is nontrivial, then \eqref{stable-Dehn} is a short exact sequence, so using that $t_{X\sharp W_\infty}$ is a central element, the 5-term exact sequence (see \cite[Corollary VII.6.4]{Brown}) of this short exact sequence has the form
%	 \[\hspace{-0.8cm}
%	 \oH_2(\pi_0(\Diff^+_\partial(X\sharp W_{\infty,1})))\ra \oH_2(\pi_0(\Diff^+(X\sharp W_\infty)))\to \bfZ/2\to\oH_1(\pi_0(\Diff_\partial(X\sharp W_{\infty,1})))\to \oH_1(\pi_0(\Diff^+(X\sharp W_\infty)))\to 0.\]
	 \[\begin{tikzcd}[row sep=0.2cm]
	   & \oH_2(\pi_0(\Diff^+_\partial(X\sharp W_{\infty,1}))) \rar{(\ast)}
	 	\ar[draw=none]{d}[name=X, anchor=center]{}
	 	& \oH_2(\pi_0(\Diff^+(X\sharp W_\infty))) \ar[rounded corners,
	 	to path={ -- ([xshift=2ex]\tikztostart.east)
	 		|- (X.center) \tikztonodes
	 		-| ([xshift=-2ex]\tikztotarget.west)
	 		-- (\tikztotarget)}]{dll}[at end]{} & \\      
	 	\bfZ/2 \rar & \oH_1(\pi_0(\Diff_\partial(X\sharp W_{\infty,1}))) \arrow[r,"(\ast)",swap] & \oH_1(\pi_0(\Diff^+(X\sharp W_\infty))) \rar & 0
	 \end{tikzcd}\]
Since $\pi_0(\Diff^+_\partial(X\sharp W_{\infty,1}))\ra \pi_0(\Diff^+(X\sharp W_\infty))$ induces a split surjection on all homology groups by the first part, the two arrows indicated by $(\ast)$ are split surjective, so by exactness  $\oH_1(\pi_0(\Diff_\partial(X\sharp W_{\infty,1})))$ is isomorphic to $\oH_1(\pi_0(\Diff^+(X\sharp W_\infty)))\oplus\bfZ/2$, as claimed.

As $\pi_0(\Diff_\partial(X\sharp W_{\infty,1})\cong\Aut(H^{\oplus \infty}) $ by the first part of the proof, the final part of the claim follows from the known fact that $\oH_1(\Aut(H^{\oplus \infty}))\cong\bfZ/2\oplus \bfZ/2$ (see e.g.\,\cite[Proposition 2.2]{GRWabelian} for a reference).
\end{proof}

In view of \cref{lem:abelianisation-lemma}, to show that $t_{X\sharp W_\infty}$ is trivial, it suffices to prove that the groups $\oH_1(\pi_0(\Diff_\partial(W_{\infty,1})))$ and $\oH_1(\pi_0(\Diff_\partial(X \sharp W_{\infty,1})))$ are abstractly isomorphic. As $\oH_1(\pi_0(G))\cong\oH_1(\mathrm{B}G)$ holds for any topological group $G$ (because $\pi_1(\mathrm{B}G)\cong \pi_0(G)$ and $\oH_1(A)$ agrees with the abelianisation of $\pi_1(A)$ for any connected space $A$), this is equivalent to showing that the first homology groups of the homotopy colimits \begin{equation}\label{equ:colim-space}\smash{\underset{g \to \infty}{\hocolim}\,\BDiff_\partial(Z \sharp W_{g,1})}\end{equation} are isomorphic in the two cases $Z=S^4$ and $Z=X=\sharp^k K_3$. Here we used that taking homology turns sequential homotopy colimits into colimits (for a reference, use the mapping telescope model for homotopy colimits and apply \cite[Section 14.6]{MayConcise}). We will show the more general statement that the homology groups of the spaces \eqref{equ:colim-space} are in any degree independent of the choice of $Z$ up to isomorphism for any closed $1$-connected spin manifold $Z$, in particular for $Z=S^4$ and $Z=X=\sharp^k K_3$. This independence result follows from work of Galatius--Randal-Williams: an application of \cite[Theorem 1.5]{GRWII} to the map $\theta\colon \mathrm{BSpin}(4)\ra \mathrm{BO}(4)$ shows that there is for any $1$-connected closed spin manifold $Z$ a map from \eqref{equ:colim-space} to a fixed infinite loop space $\Omega^\infty \mathrm{MT\theta}$, independent of $Z$, that is a homology equivalence onto the path component of $\Omega^\infty \mathrm{MT\theta}$ that is hit by this map; which path component may depend on $Z$. Deducing this from the cited result uses that, by an application of obstruction theory, the spaces $\BDiff_\partial(W)$ and $\mathrm{Bun}_{n,\partial}^\theta(TW;\hat{\ell}_P)\sslash \Diff_\partial(W)$ in Definition 1.1 loc.cit.~are homotopy equivalent for $n=2$, $P=S^3$, and $W=Z \sharp W_{g,1}$. Since all path components of a loop space are homotopy equivalent to the component of the constant loop (by multiplication with the inverse loop), all path components of $\Omega^\infty \mathrm{MT\theta}$ have the same homology groups, so the claim follows.
%
%and taking homology turns sequential homotopy colimits into colimits (for a reference, use the mapping telescope model for homotopy colimits and apply \cite[Section 14.6]{MayConcise}), we conclude 
%\begin{equation}\label{equ:MT-iso}\underset{g \to \infty}{\colim}\, \oH_k(\BDiff_\partial(Z \sharp W_{g,1})) \cong\oH_k(\Omega_0^\infty \mathrm{MT\theta})\end{equation} for all $k\ge0$ where the subscript $0$ in $\Omega_0^\infty$ indicates the path component of the constant loop. In particular, the left-hand side is independent of the choice of $Z$. Specialising this to $k=1$ and $Z=S^4$ as well as $Z=X=\sharp^k K_3$ and combining it with the discussion around \eqref{equ:splittin-iso}, this implies \cref{thm:dehn} once we show that the groups \eqref{equ:MT-iso} are finitely generated. 
%
%
%To show that the homology of $\Omega_0^\infty \mathrm{MT\theta}$ is indeed degreewise finitely generated, it suffices to show that this holds for the homology of the spectrum $\mathrm{MT\theta}$ by standard Serre class arguments (\cite[Chapter III.2]{Adams} for homology and homotopy groups of spectra, \cite[Section 9.6]{Spanier} for Serre classes, and \cite[proof of Theorem 5.1]{Kupers} for a similar argument). But by the Thom isomorphism, we have $\oH_k(\mathrm{MT\theta})\cong \oH_{k+4}(\mathrm{BSpin}(4))$ and as $\mathrm{BSpin}(4)$ has finitely generated homology groups, the result follows.

\begin{rem}This method can be used to resolve the ambiguity in \cite[Thm 7.1]{Giansiracusa}: the kernel is $(\bfZ/2)^{n-1}$.\end{rem}

\begin{rem}One can provide an explicit bound on the number of stabilisations by $S^2 \times S^2$ that are necessary to make $t_X$ in \cref{thm:dehn} \emph{topologically} isotopic to the identity, using the classification of such manifolds up to homeomorphism after a single stabilisation \cite[Section 10.1]{FreedmanQuinn} and the fact $t_{K3}$ is \emph{topologically} trivial (this implies the same for $t_{\overline{K3}}$). The latter uses the following reformulation (see \cite[Section 2]{KronheimerMrowka}): $t_{K3}$ is smoothly (or topologically) nontrivial if and only if every smooth (or topological) bundle $K3 \to E \to S^2$ satisfies $w_2(TE) = 0$. The existence of a topological bundle with $w_2(TE) \neq 0$ follows from \cite[Section 5]{BaragliaKonno}; it is the ``difference'' of the bundles $E$ and $E'$ over $S^1\times S^1$ which agree over the $1$-skeleton.\end{rem}

\subsection*{Acknowledgments} We thank David Gay for a conversation regarding the first part of this note, Peter Kronheimer for answering some questions related to \cite{KronheimerMrowka}, Mark Powell for pointing us to \cite{Saeki}, and the anonymous referee for their comments and suggestions.

\vspace{-0.1cm}
\bibliographystyle{amsalpha}
\bibliography{literature}

\end{document}